\documentclass[12pt,reqno]{amsart}
\usepackage{hyperref}
\usepackage{amsmath,amssymb}
\usepackage[english]{babel}
\usepackage{caption}
\usepackage{amssymb,amsmath,euscript,enumerate,tikz}
\usepackage[margin=1in]{geometry}
\usepackage{graphicx}
\usepackage{float}
\usepackage{pgf,tikz}
\usepackage{mathrsfs}
\usepackage{upgreek}
\newcommand{\sk}{{\ensuremath{\sf k }}}

\DeclareMathOperator{\Tor}{Tor}

\newtheorem{conjecture}{Conjecture}[section]
\newtheorem{theorem}[conjecture]{Theorem}
\newtheorem{lemma}[conjecture]{Lemma}

\newtheorem{proposition}[conjecture]{Proposition}
\theoremstyle{definition}
\newtheorem{remark}[conjecture]{Remark}

\newtheorem{question}{Question}

\newtheorem{definition}[conjecture]{Definition}
\newtheorem{notation}[conjecture]{Notation}
\newtheorem{example}[conjecture]{Example}

\providecommand\ass{\text{\rm Ass}}

\renewcommand\dim{\text{\rm dim}}

\providecommand\max{\text{\rm max}}

\providecommand\Tor{{\rm Tor}}
\providecommand\reg{{\rm reg}}

\begin{document}
\title{Regularity comparison of symbolic powers, integral closure of powers and powers of edge ideals}
\author[Arvind Kumar]{Arvind Kumar$^{1, 2}$}
\email{arvindkumar@cmi.ac.in}
\thanks{$^1$ The author is partially supported by Sciences and Engineering Research
Board, India under the National Postdoctoral Fellowship (PDF/2020/001436).}
\thanks{$^2$ The author is also partially supported by the Infosys Foundation.}
\address{Department of Mathematics, Chennai Mathematical Institute, Siruseri
Kelambakkam, India -603103.}
\author[Rajiv Kumar]{Rajiv Kumar}
\email{rajiv.kumar@iitjammu.ac.in}
\address{Department of Mathematics, Indian Institute of Technology Jammu, J\&K, India - 181221.}

\begin{abstract}
We study the regularity of small symbolic powers and integral closure of small powers of edge ideals. We also prove that the regularity of integral closure of powers of edge ideals of graphs with at most two odd cycles is the same as the regularity of their powers. 
\end{abstract}

\subjclass[AMS Subject Classification (2020).]{Primary: 13D02,  13B22, 13F55, Secondary: 05E40}


\maketitle
\section{Introduction} 
Let $R=\sk[x_1,\ldots,x_n]$ be a standard graded polynomial ring over a field $\sk$, and $I$ be a homogeneous ideal of $R$. One of the important invariants associated with $I$ is the \textit{Castlenuvo-Mumford regularity (simply regularity)} of $I$ (see Section \ref{sec-2} for the definition).  The study of the regularity of homogeneous ideals emerged as an active area of research.  Many researchers have studied the regularity of homogeneous ideals and their powers in the last two decades. One of the important problems in this direction is  to find the regularity  functions $\reg(I^s)$, $\reg(\overline{I^s})$ and $\reg(I^{(s)}) $, where $\overline{I^s}$ denote the integral closure of $I^s$ and $I^{(s)}$ denote the $s$-th symbolic power of $I$ (see Section \ref{sec-2} for the definition).

Cutkosky, Herzog and Trung \cite{CHT} and independently Kodiyalam \cite{vijay} proved that  $\reg(I^s)$ and $\reg(\overline{I^s})$ are eventually linear functions. However, $\reg(I^{(r)})$ is not eventually linear in general, as Catalisano, Trung and Valla \cite[Proposition 7]{CTV93}, proved that if $I$ defines $2q+1$ points on a rational normal curve in $\mathbb{P}^q$, $q \geq 2$, then for all $r \geq 1$, $\reg(I^{(r)})=2r+1+\lfloor \frac{r-2}{q} \rfloor$.  However, for the case of edge ideals of graphs, if Minh's conjecture holds, then $\reg(I^{(r)})$ is an eventually linear functiion.  Minh (see \cite[p.1]{GHJsymbo}) conjectured that for any graph $G$ on the vertex set $\{x_1,\ldots,x_n\}$,  $\reg(I(G)^{(s)})=\reg(I(G)^s)$ for all $s \geq 1$, where $I(G)$ denote the \textit{edge ideal} of $G$ that is generated by $\{x_i x_j \mid \{x_i , x_j \} \in E (G) \}$. Minh's conjecture trivially holds for bipartite graphs due to a result of  Simis, Vasconcelos and Villarreal, (see  \cite{SVV1}), and it has been verified for various classes of non-bipartite graphs by many researchers, for example, Gu et al. \cite{GHJsymbo} verified it for odd cycles, 
Jayanthan and Kumar \cite{JK19} proved it for the clique sum of an odd cycle with certain bipartite graphs, Fakhari proved it for chordal graphs in \cite{Seyed-chordal} and for unicyclic graphs in  \cite{fak-uni},  Kumar et al.  \cite{KKS19} proved this conjecture for complete multipartite graphs and wheel graphs, Kumar and Selvaraja \cite{KS19} and independently Fakhari \cite{SFCW} verified it for Cameron-Walker graphs. Considering the Minh's conjecture and the fact that $I(G)^s \subseteq \overline{I(G)^s} \subseteq I(G)^{(s)}$ for each $s$, one may ask the following question: 
\begin{question}\label{que}
    Let $G$ be a finite simple graph. Does  $$ \reg(I(G)^s)=\reg(\overline{I(G)^s})=\reg(I(G)^{(s)}) \text{ for all } s \ge 1?$$
    More generally, does $ \reg(I(G)^s)=\reg(J)$ whenever  $J$ is a monomial ideal with $I(G)^s \subseteq J \subseteq I(G)^{(s)}?$
\end{question}
  
The following examples demonstrate that the above question generally has a negative answer.
\begin{example}
Let	 $R=\sk[x,y]$ and $I=( x^2, y^2) $. Then, $I^{(s)}=\overline{I^s}=(x,y)^{2s}$ for all $s \geq 1$. Therefore, $\reg\left(\overline{I^s}\right) =2s$ for all $s \geq 1$. Since $I$ is a complete intersection ideal, it follows from \cite[Lemma 4.4]{BHT} that  $\reg(I^s)=2s+1$ for all $ s \geq 1$. Hence, $\reg(I^s) \neq \reg(\overline{I^s})$ for all $s \geq 1.$
\end{example} 
Even if we restrict our attention to squarefree monomial ideals, the above question has a negative answer. 
\begin{example}
Let $R=\sk[x_1,\dots, x_6]$ and $$I=( x_1x_4x_5,x_1x_3x_6,x_2x_3x_4,x_2x_5x_6,x_3x_4x_5,x_3x_4x_6,x_3x_5x_6,x_4x_5x_6). $$ Since $I$ is radical ideal, $I=\overline{I}$, and hence, $\reg(I)=\reg{\left(\overline{I}\right)}$. Using Macaulay2 \cite{M2}, we obtain $\reg\left(I^2\right)=7$ and $\reg\left(I^{(2)} \right)=\reg{\left(\overline{I^2}\right)}=6$. Therefore, $\reg(I^2) \neq \reg(\overline{I^2}) $.
\end{example}

Many researchers have explored the regularity of powers of edge ideals of simple graphs; see \cite{AB, ASH20, AE20, BHT, JNS18,  AS21} and references therein. Also, the regularity function $\reg\left( I(G)^{(s)}\right)$ has been explored by many researchers, see \cite{GHJsymbo, JK19, KKS19, KS19, Seyed-chordal, SFCW, fak-uni}. However, there are no results on the behaviour of the regularity function $\reg\left(\overline{I(G)^s}\right).$ Our aim in this article is to start the study of the regularity function $\reg\left( \overline{I(G)^s} \right)$, and to see how it is behaves as compare to $\reg(I(G)^s)$ and $\reg\left( I(G)^{(s)}\right)?$ We begin with comparing  $\reg(I(G)^s)$, $\reg\left(\overline{I(G)^s}\right)$ and $\reg\left( I(G)^{(s)}\right)$ when $s$ is a small and $G$ is any graph. Recently, in \cite{MNPTV}, it was shown that for any graph $G$, $\reg\left( I(G)^{(s)}\right)=\reg(I(G)^s)$ for $s \le 3$. We prove a more general result; consequently, we recover one of the main results of \cite{MNPTV}. We prove the following: 

\medskip

\noindent\textbf{Theorem \ref{main-sym-thm}.} Let $G$ be a non-bipartite graph. Suppose that smallest induced odd cycle in $G$ has size ${2n+1}.$ Then, $$\displaystyle \reg\left(I(G)^{(s)} \right) =\reg (I(G)^s) \text{ for } s \leq n+1.$$

\medskip

Simis, Vasconcelos and Villarreal in \cite{SVV1} proved that  $I(G)$ is normal  if and only if $G$ is bow-free (see Section \ref{sec-2} for the definition). Therefore, if $G$ has at most one odd cycle, for example, bipartite graphs, unicyclic graphs, etc., then $I(G)$ is normal, and hence,  $\reg\left(\overline{I(G)^s}\right)=\reg(I(G)^s)$ for all $s \geq 1$. Recently, in \cite{CV21}, it was shown that, $\displaystyle \reg\left(\overline{I(G)^{s}}\right) = \reg\left(I(G)^{s}\right)$ for $s \le 4$. We prove a more general result and as a consequence, we recover the main result of \cite{CV21}. We prove the following:

\medskip

\noindent\textbf{Proposition \ref{int-decomposition} and Theorem \ref{k+1}.}  If the smallest size of a bow in $G$ is $k+1$, then $$\displaystyle \reg\left(\overline{I(G)^{s}}\right) = \reg\left(I(G)^{s}\right) \text{ for } s \le k+1.$$ 

\medskip

As we mentioned above,  if $G$ has at most one odd cycle, for example, bipartite graphs, unicyclic graphs, etc., then $I(G)$ is normal, and hence,  $\reg\left(\overline{I(G)^s}\right)=\reg(I(G)^s)$ for all $s \geq 1$. Thus, to study the equality between the regularity functions $\reg\left(\overline{I(G)^s}\right)$ and $\reg(I(G)^s)$, one should first consider the class of odd bicyclic graphs.   A graph is said to be an {\it odd bicyclic graph} if it has exactly two odd cycles. In this case, we prove the two regularity functions are the same. Precisely, we prove the following:

\medskip

\noindent\textbf{Theorem \ref{mainTheorem}.} Let $G$ be an odd bicyclic graph. Then,  $$\displaystyle\reg\left(\overline{I(G)^s}\right)= \reg\left(I(G)^s\right) \text{ for all } s.$$
	
\medskip

The article is organized as follows: We recall definition and notation in the second section. In Section \ref{tech-res}, we obtain some technical results that we will use in subsequent sections. In Section \ref{small-sym}, we study the regularity of small symbolic powers of edge ideals. In Section \ref{small-int}, we study the regularity of integral closure of small powers of edge ideals. In Section \ref{bicyclic-int}, we prove that $\reg\left(\overline{I(G)^{s}}\right) = \reg\left(I(G)^{s}\right)$ for all $s \ge 1$ if $G$ is a bicyclic graph.

\section{Preliminaries}\label{sec-2}
 In this section, we recall definition and notation, which we use in later sections. For undefined terminologies, we refer the readers to the book  \cite{villarreal_book}.  

Let $R=\sk[x_1,\dots, x_n]$	be a standard graded polynomial ring over a field $\sk$, and $I$ be a homogeneous ideal of $R$.  The $s$-th \emph{symbolic power} of $I$ is  defined as: $$I^{(s)}=\bigcap\limits_{\mathfrak{p} \in \ass{(I)}} \left(I^sR_{\mathfrak{p}}\cap R\right),$$ where $\ass(I) $ is the set of associated primes of $I$. The \emph{symbolic Rees algebra } of $I$ is defined as: $$\mathcal{R}_s(I)=\bigoplus\limits_{s=0}^\infty I^{(s)}t^s.$$

 A homogeneous element $r\in R$ is said to be  \emph{integral over} $I$ if there exist  homogeneous elements $a_1, \ldots,a_n \in R$ such that $$r^n+a_1r^{i-1}+\dots+a_{n-1}r+a_n=0$$ and  $a_i\in I^i$ for  $1 \leq i \leq n.$ The set of all elements that are integral  over $I$ is called \emph{integral closure} of $I$, and it is denoted by $\overline{I}$. An ideal $I$ is said to be \emph{integrally closed} if $I=\overline{I}$. If  $I^i$ is integrally closed for all $i\geq 1$, then we say that $I$ is  {\it normal}. 
 The  \emph{Rees algebra and normal Rees algebra of $I$} are algebras $\mathcal{R}[It]=\bigoplus\limits_{i=0}^\infty I^it^i$ 
 and $\overline{\mathcal{R}[It]}=\bigoplus\limits_{i=0}^\infty \overline{I^i}t^i$, respectively.
 
 Let $M$ be a finitely generated non-zero graded $R$-module. For $ i,j \geq 0$,  the $(i,j)$-th \emph{graded Betti number of $M$}, denoted by $\beta_{i,j}^R(M)$, is $\dim_{\sk}\left(\Tor_i^R(M,\sk)_j\right).$ The \emph{Castlenuvo-Mumford regularity of} $M$,  denoted by $\reg(M)$, is  $$\max\left\{j-i \; : \;  \beta_{i,j}^R(M) \neq 0 \right\}.$$ Observe that if $I$ is a non-zero proper homogeneous ideal of $R$, then $\reg(I) =\reg(R/I)+1.$ Also, we make a convention that if $M$ is a zero module, then $\reg(M)=-\infty.$

The following lemma on regularity is used repeatedly in this article. We refer the reader to the book \cite[Chapter 18]{P11} for more properties on regularity.
 
\begin{lemma}\label{reg-lem}
 Let $R$ be a standard graded ring, $M,N$ and $P$ be finitely generated graded $R$-modules. 
 If $ 0 \rightarrow M \xrightarrow{f}  N \xrightarrow{g} P \rightarrow 0$ is a short exact sequence with $f,g$  
 graded homomorphisms of degree zero, then 
 \begin{enumerate}[\rm a)]
  \item $\reg(P) \leq \max\{\reg(N),\reg(M)-1\}$.
  \item $\reg(P) = \reg(N)$, if $\reg(M) <\reg(N)$. 
 \end{enumerate}
\end{lemma}

Now, we recall notation and definition from graph theory.
	\begin{enumerate}[i)]
	\item A graph $G$ is called an \emph{empty graph} if $E(G)=\emptyset$, otherwise it is called a \emph{non-empty graph}.  
 \item Let $G$ be a graph on the vertex set $V(G)$ and  $W\subset V(G)$. Then a graph $H$ on the vertex set $W$ is called a \emph{subgraph of $G$} if  edge set $E(H)\subset E(G)$.
	\item Let $G$ be a graph on the vertex set $V(G)$ and  $W\subset V(G)$. The \emph{induced subgraph of $G$ on $W$}, denoted by $G[W]$, is a graph on the vertex set $W$ and edge set $$E(G[W])=\{x_ix_j \; : \; x_i,x_j\in W \text{ and }  x_ix_j \in E(G)\}.$$ A subgraph $H$ of $G$ is said to be an {\it induced subgraph } of $G$ if $H=G[W]$ for some $W \subset V(G)$.
		\item A graph on $n$ vertices is called an \emph{$n$-path} if there exists a labeling of  vertices such that its edge set is $\{x_1x_2, x_2x_3, \dots, x_{n-1}x_n\},$ and it is denoted by $P_n$.
		\item A \emph{walk} in $G$ is a sequence  $x_0, x_1, \ldots , x_k$ of vertices such that $x_{i-1}x_i \in E(G)$  for each $1 \leq i \leq k$. 
		\item A graph on $n$ vertices is called an \emph{$n$-cycle} if there exists a labeling of  vertices such that its edge set is $\{x_1x_2, x_2x_3, \dots, x_{n-1}x_n, x_nx_1\},$ and it is denoted by $C_n$.
		\item A \emph{tree} is a connected graph with no induced cycle, and a \emph{forest} is a disconnected graph with no induced cycle. 
		\item A \emph{chordal graph} is a graph that does not have an induced cycle on $n$ vertices with  $n \geq 4$. 
		\item A graph $G$ is called a \emph{bipartite graph} if there exists a partition of $V(G)=V_1\sqcup V_2$ such that $G[V_1]$ and $G[V_2]$ are empty graphs, otherwise it is called a \emph{non-bipartite graph}.
		\item A collection of edges $\{e_1,\ldots, e_s\}$ of $G$ is called an \emph{induced matching of $G$} if the induced subgraph of $G$ on the vertex set $\cup_{j=1}^s e_j$ is a disjoint union of edges. The \emph{induced matching number of $G$}, denoted by $\nu(G)$, is defined as follows:
		$$\nu(G)=\max\{s: \{e_1,\dots, e_s\} \text{ is an induced matching of } G\}.$$
		\item\label{bow} For $u,v \in V(G)$, the \emph{distance between $u$ and $v$}, denoted by $d_G{(u,v)}$, is defined as $$d_G(u,v)=\min\{\text{ number of edges in } P \; : \;  P \text{ is a path between } u \text{ and } v\}.$$
		\item A {\it bow} is a graph consisting of two odd cycles $C$ and $C'$ such that the distance between $C$ and $C'$ is atleast two. A {\it bow-free graph} is a graph that does not have an induced bow.
		\item For subgraphs $H, H'$ of $G$, the {\it distance between $H$ and $H'$}, denoted by $d_G(H,H')$, is $\min \{ d_G(u,v) : \; u \in V(H) \text{ and } v \in V(H') \}$. 
		\item For $T\subset V(G)$, $N_G(T) = \{ x  : xy \in E(G) \text{ for some }  y \in T\}. $
  	\end{enumerate}
 All graphs considered in this article are finite and simple.

\section{Technical Results}\label{tech-res}
In this section, we obtain some technical results to prove the equality of regularity of small symbolic powers, integral closure powers, and ordinary powers of edge ideals in the subsequent sections. For that purpose, we recall the definition of even-connection introduced by Banerjee in \cite{AB}.

\begin{definition}\label{even-con}{\rm
		Let $G$ be a graph and $x, y \in V(G)$. Let $u=e_1 \cdots e_s \in I(G)^s$ for some $e_1,\ldots, e_s \in E(G)$. Then, $x$ and $y$ are {even-connected} with respect to $u=e_1\cdots e_s$ if there is a walk $p_0p_1\cdots p_{2k+1}$, $k\geq 1$ in $G$ such that 
		\begin{enumerate}[i)]
			\item $p_0=x$ and $p_{2k+1}=y$.
			\item For all $1\leq l\leq k$, we have $p_{2l-1}p_{2l}=e_i$ for some $i$.
			\item For all $i$, we have $|\{l\geq 1:p_{2l-1}p_{2l}=e_i \}|\leq|\{j:e_j=e_i\}|.$
		\end{enumerate}
}\end{definition}

Next, we fix the notation that we use throughout this article. 

\begin{notation}\label{notation}
	Let $G$ be a graph. For $T \subset V(G)$, we set   $W_{T}=N_G(T)$, $H_T=G[{V(G) \setminus W_T}]$, and  $\displaystyle m_{T}=\prod_{x\in T}x$.
\end{notation}

Now, we obtain a few auxiliary lemmas.  

\begin{lemma}\label{short-lemma}
Let $G$ be a graph and $T \subset V(G)$. If $m_T \not\in I(G)^k$ and $x m_T \in I(G)^k$ for all $x \in W_T$, then  $$I(G)^{k}:m_{T} =(w  : w \in W_T)+I(H_T).$$
\end{lemma}
\begin{proof}
 Let $x \in W_T$. Since $xm_T \in I(G)^k$, there exist  $e_1,\ldots, e_k \in E(G)$ and a monomial $m'$ such that $xm_T = e_1\cdots e_k m'$. Then, $x$ divides some $e_i$ as $m_T \not\in I(G)^k$.  Assume, without loss of generality, that $e_k=xz$ for some $z \in V(G)$. Therefore, $m_T=e_1\cdots e_{k-1} zm'$ and $z \in T$. This implies that  $\displaystyle I(G)^{k}:m_{T} =\left(I(G)^{k}:e_1 \cdots e_{k-1}\right):zm'$.  By \cite[Theorem 6.7]{AB}, $$\displaystyle I(G)^{k}:e_1\cdots e_{k-1}=I(G)+\left(uv \; : \; u \text{ is even-connected to } v \text{ with respect to } e_1\cdots e_{k-1}\right).$$ Let  $u$ be  even-connected to $v$ with respect to  $e_1\cdots e_{k-1}$. By Definition \ref{even-con},  $u,v \in N_G(T)=W_T$. Thus, $\displaystyle I(G)^{k}:e_1\cdots e_{k-1} \subset I(G)+\left(uv: u,v \in W_T \right),$ and hence,
\begin{align*}
I(G)^{k}:m_{T} & =\left(I(G)^{k}:e_1\cdots e_{k-1}\right):zm'\\& \subset I(G):zm'+(uv: u,v \in W_T ):zm'\\& = I(G)+\left(y \; : \; yy'  \in E(G) \text{ and } y' \text{ divides } zm'\right)+ (w :\; w \in W_T)\\&=I(G)+(w  : w \in W_T),
\end{align*} where the last equality follows from the fact that if $y'$ divides $zm'$, then $y' \in T$, and therefore, $y \in W_T.$  Clearly,  $(w : w \in W_T) \subset I(G)^{k}:m_{T}$. Since $I(G)m_{T} \subset I(G)^{k}$, we get  $I(G)^{k}:m_{T} =(w  : w \in W_T)+I(G)=(w  : w \in W_T)+I(H_T)$. Hence, the assertion follows.
\end{proof} 

\begin{lemma}\label{new-lower-lem}
Let $G$ be a graph and $T \subset V(G)$ such that $G[T]$ is a non-empty graph. Then for all $s\geq 1,$ 
\begin{align*}
\reg(I(G)^s) \geq 2s+\nu(G[T])+\max\{\reg(I(H_T)),1\} -2 
\end{align*}
\end{lemma}
\begin{proof}
First, note that $H_T\sqcup G[T]$ is an induced subgraph of $G$. Therefore, by \cite[Corollary 4.3]{BHT}, $\reg\left(I(G)^s\right) \geq \reg\left(I(H_T\sqcup G[T])^s\right)$ for all $s\geq 1$. Thus, for all $s\geq 1,$ it is enough to prove that 
$ \reg\left(I(H_T\sqcup G[T])^s\right) \geq 2s+\nu(G[T])+\max\{\reg(I(H_T)),1\} -1 .$ If $H_T$ is an empty graph, then $I(H_T\sqcup G[T])=I(G[T])$, and hence,   $\reg(I(H_T \sqcup G[T])^s)=\reg(I(G[T])^s)\geq 2s+ \nu(G[T])-1,$ by \cite[Theorem 4.5]{BHT}. 

Next, we assume that $H_T$ is a non-empty graph. Thus, for $s\geq 1$, it follows from \cite[Theorem 1.1]{NV19} that \begin{align*}
&\reg\left(I(H_T\sqcup G[T])^{s}\right)=\reg\left((I(H_T)+I(G[T]))^{s}\right)\\&= \max_{\substack{1 \leq i \leq s-1\\ 1 \leq j \leq s}} \Big\{\reg\left(I(H_T)^{s-i}\right) +\reg\left(I(G[T])^{i}\right), \reg\left(I(H_T)^{s-j+1}\right)+\reg\left(I(G[T])^j\right)-1 \Big\}\\& \geq    \reg\left(I(H_T)\right)+\reg\left(I(G[T])^s\right)-1\\&\geq  \reg(I(H_T))+2s+\nu(G[T])-2, 
\end{align*} where the first inequality follows by using $j=s$ in {the} previous equality, and the second inequality follows by using \cite[Theorem 4.5]{BHT}. Hence, the assertion follows.
\end{proof}

We prove the main result of this section that we use in subsequent sections.

\begin{theorem}\label{gen}
Let $G$ be a graph and $k \in \mathbb{N}.$ Let $T_1, \ldots, T_r \subset V(G)$ be such that for each $i$
\begin{enumerate}
    \item $G[T_i]$ is a non-empty graph with no isolated vertex;
    \item $m_{T_i} \notin I(G)^k$ and $xm_{T_i} \in I(G)^k$ for all $x \in W_{T_i}$;
    \item $ |T_i| \leq 2k+\nu(G[T_i])-2;$
    \item $2k+\nu(G[T_i])-2 +\max\{1, \reg(I(H_{T_i})) \} \leq \reg\left( I(G)^k+(m_{T_1},\ldots,m_{T_r})\right).$
\end{enumerate} Then, $$ \reg\left( I(G)^k+(m_{T_1},\ldots,m_{T_r})\right) = \reg(I(G)^k).$$
\end{theorem}
\begin{proof}
Assume, without loss of generality, that $m_{T_i} \notin (m_{T_j})$ {for} $i \neq j$. Set $I_0=I(G)^{k}$, and for $1\leq j\leq r$, $I_j=I(G)^{k}+ ( m_{T_1},\dots , m_{T_j})$.  We need to prove that $\reg(I_0)=\reg(I_{r})$. For $1 \leq j \leq r$, consider the following short exact sequences:
 \begin{align*}
    0\rightarrow [{I_{j-1}:m_{T_{j}}}](-|T_j|) \xrightarrow{\cdot m_{T_{j}}} {I_{j-1}} \rightarrow {I_j} \rightarrow 0.
\end{align*} Next, we compute $I_{j-1}:m_{T_j}$ for $1 \leq j \leq r.$ By Lemma \ref{short-lemma}, we know that $$I_0:m_{T_1} =(w \; : \; w \in W_{T_1})+I(H_{T_1}).$$ So assume that $1 < j \leq r.$ Then, \begin{align*}
    I_{j-1}: m_{T_j} &=\left(I(G)^{k}+ ( m_{T_1},\dots , m_{T_{j-1}}) \right) : m_{T_j}\\&=I(G)^{k}:m_{T_j}+ ( m_{T_1},\dots , m_{T_{j-1}}):m_{T_j} \\ & =(w \; : \; w \in W_{T_j})+I(H_{T_j}) + ( m_{T_1},\dots , m_{T_{j-1}}):m_{T_j},
\end{align*} where the last equality follows from Lemma \ref{short-lemma}. Note that \begin{align*}
  ( m_{T_1},\dots , m_{T_{j-1}}):m_{T_j} &= \left( m_{ T_i \setminus T_j} \; : \; 1 \leq i \leq j-1\right).   
\end{align*} If $m_{T_i \setminus T_j}\in I(G[T_i])$, then $m_{T_i \setminus  T_j} \in I(G) \subset (w \; : \; w \in W_{T_j})+I(H_{T_j}).$ Suppose that $m_{T_i \setminus T_j} \notin I(G[T_i]).$ Since $m_{T_j} \notin (m_{T_i})$, we get $m_{ T_i \setminus T_j}\neq 1$. Therefore, there exist variable $u\in T_i\setminus T_j$ and $w \in T_j$ such that $uw \in E(G[T_i])$ as $G[T_i]$ is a non-empty graph.  So, in this case,  $m_{ T_i \setminus T_j} \in (w \; : \; w\in W_{T_j}),$ as $u \in W_{T_j}.$ Thus, for each $1 \leq j \leq r$, $( m_{T_1},\dots , m_{T_{j-1}}):m_{T_j} \subset (w \; : \; w \in W_{T_j})+I(H_{T_j})$, and hence, $I_{j-1}:m_{T_j}=(w \; : \; w \in W_{T_j})+I(H_{T_j})$. Therefore, for each $1 \leq j \leq r$, \begin{align*}
    \reg(I_{j-1} :m_{T_j} (-|T_j|)) & = |T_j| +\reg(I_{j-1}:m_{T_j})\\& =|T_j|+ \reg((w \; : \; w \in W_{T_j})+I(H_{T_j})) \\ & =|T_j| + \max\{1, \reg(I(H_{T_j})) \}\\& \leq 2k+\nu(G[T_j])-2 +\max\{1, \reg(I(H_{T_j})) \}\\& \leq  \delta, 
\end{align*} where $\delta=\max\{2k+\nu(G[T_j])-2 +\max\{1, \reg(I(H_{T_j})) \}\; : \; 1 \leq j \leq r\}.$  Suppose  that $\reg(I(G)^{k}) > \delta.$ Then, by recursively  applying  Lemma \ref{reg-lem} on short exact sequences, we get  \begin{align*}
	  \reg\left(I_0\right)=\reg\left(I_{1}\right)=\cdots =\reg\left(I_r\right).  
	\end{align*} Next, we assume that $\reg(I(G)^{k}) \leq \delta .$ Then,  by Lemma \ref{new-lower-lem}, we get that $\reg(I(G)^{k}) = \delta.$ By applying Lemma \ref{reg-lem} on short exact sequences, we get
	\begin{align*}
	  \reg\left(I_r\right) & \leq \max\Big\{ \reg\left(I_{r-1}\right), \reg\left( I_{r-1}:m_{T_r} (-|T_r|) \right)-1 \Big\} \\ & \leq \max\Big\{ \reg\left(I_{r-2}\right), \reg\left( I_{r-2}:m_{T_{r-1}} (-|T_{r-1}|) \right)-1, \reg\left( I_{r-1}:m_{T_r}(-|T_r|) \right)-1 \Big\} \\ & \leq \cdots \cdots \cdots \cdots \cdots \cdots (\text{ continuing this process })\\ & \leq \max\Big\{ \reg\left(I_{0}\right), \reg\left( I_{i-1}:m_{T_i}(-|T_i|) \right)-1 \text{ for } 1 \leq i\leq  r \Big\} \\ & \leq \delta= \reg(I_0).  
	\end{align*} Now, by hypothesis, $\reg(I_r)  \geq \delta.$ Thus, $\reg(I_0)=\reg(I_r). $ Hence, the assertion follows.
\end{proof}

\section{Regularity of small symbolic powers of edge ideals}\label{small-sym}
In this section, we prove Minh's conjecture for the smallest $s$ for which $I(G)^{(s)} \neq I(G)^s$.  
We start by computing a monomial generating set of small symbolic powers of edge ideals. Bernal et al. in \cite{BRC2011} obtained a minimal generating set of symbolic Rees algebra of edge ideals of clutter (hypergraphs) in terms of vertex cover number of duplication and parallelization of clutter. We first recall here these terminologies and some of their properties.

\begin{definition}{\rm Let $G$ be a graph on the vertex set $V$. Then,
\begin{enumerate}[i)]
\item the \emph{vertex cover number} of $G$, denoted by $\tau(G)$, is $\min\{|C|: C \text{ is a vertex cover of } G\}$;
\item $G$ is called \emph{decomposable} if there
exists a partition of $V=V_1\sqcup\cdots\sqcup V_r$ such that
$\sum_{i=1}^r \tau(G[V_i])=\tau(G)$. If $G$ is not decomposable, then
$G$ is called \emph{indecomposable};
%\item for $T \subset V(G)$, $G \setminus T$ denote the induced subgraph of $G$ on the vertex set $V \setminus T$;\textcolor{red}{why new notation?}
%\item a subset $T \subset V(G)$ is called a \emph{separated set} if the number of connected components of  $G\setminus T$ is more than the number of connected components of $G$.
\end{enumerate}
}\end{definition}
 
\begin{definition}{\rm Let $G$ be a graph on the vertex set $\{x_1,\ldots,x_n\}$ and ${\bf
  v}=(v_1,\dots, v_{n})\in\mathbb{Z}^{n}_{\geq 0}.$
	\begin{enumerate}[\rm i)]
		\item The \emph{duplication} of $G$ with respect to a vertex $x$ of $G$ is the graph obtained from $G$ by adding a new vertex $x'$ and  edges $\{x'y \;:\; y\in N_G(x)\}$.
		\item The \emph{parallelization} of $G$ with respect to ${\bf
		  v}$, denoted by $G^{\bf v}$, is the graph obtained from $G$
		  by deleting $x_i$ if $v_i= 0$ and duplicating $v_i-1$ times
		  $x_i$ if $v_i\geq 1$.
	\end{enumerate}
}\end{definition}

For a vector $\mathbf{v} \in \mathbb{Z}^n_{\geq 0}$, $\mathbf{x}^{\mathbf{v}}$ denote the monomial $x_1^{v_1}\cdots x_n^{v_n}$ in the polynomial ring $ \sk[x_1, \ldots, x_n]$.
We now obtain a monomial generating set of small symbolic powers of edge ideals.

\begin{theorem}\label{sym-decomp}
Let $G$ be a non-bipartite graph. Suppose that smallest induced odd cycle in $G$ has size ${2n+1}.$ Then, we have the followings:
\begin{enumerate}
    \item $I(G)^{(s)}= I(G)^s$ for $s \leq n;$
    \item $I(G)^{(n+1)}= I(G)^{n+1}+ (m_C : C \text{ is an induced odd cycle in } G \text{ of size } 2n+1).$
    %\item $I(G)^{(n+2)}=I(G)^{n+1}+ (m_C : C \text{ is an induced odd cycle in } G \text{ of size } 2n+1)I(G) \\ \text{ } \hspace{4cm} + (m_C : C \text{ is an induced odd cycle in } G \text{ of size } 2n+3)\\ \text{ } \hspace{4cm}+ (x_{i_1}\cdots x_{i_{4}} \; : \; G_{\{x_{i_1}, x_{i_2},x_{i_3},x_{i_{4}}\}} \text{ is a complete graph}).$
\end{enumerate}
\end{theorem}
\begin{proof} (1) Follows from \cite[Theorem 4.13]{DDAGHN}.
\par (2)   It follows from \cite[Theorem 2.6]{BRC2011} that $$\mathcal{R}_s(I(G))=\sk[\{x^{\bf v}t^b: G^{\bf v} \text{ is an indecomposable graph and } b =\tau(G^{\bf v})\}].$$ By part $(1)$, $I(G)^{(s)}= I(G)^s$ for $s \leq n$ and by \cite[Theorem 4.13]{DDAGHN}, $I(G)^{n+1}\neq I(G)^{(n+1)}.$ Consequently, we have  $I(G)^{(n+1)}=I(G)^{n+1}+J$, where $$J=(\{x^{\bf v}: G^{\bf v} \text{ is an indecomposable graph and }  \tau(G^{\bf v})=n+1\}).$$ Let $G^{\bf v}$ be an indecomposable graph with $\tau(G^{\bf v})=n+1$. Then, by \cite[Corollary 1b]{HP}, $G^{\bf v}$ contains an odd cycle.  Let $i_1<\cdots< i_r$ such that $v_{i_j}>0$ for $1 \leq j \leq r$ and $v_i =0$ if $i \notin \{i_1,\ldots, i_r\}$. Let $H$ be the induced subgraph of $G$ on the vertex set $\{x_{i_1},\ldots, x_{i_r}\}$ and ${\bf w} \in \mathbb{Z}^{r}_{\geq 0}$ such that $w_j = v_{i_j}$ for all $1\leq j \leq r.$ Then note that $G^{\bf v} =H^{\bf w}.$ Since parallelization of a bipartite graph is bipartite and $H^{\bf w}$ is a non-bipartite graph, we get that $H$ is a non-bipartite graph. Let $C$ be a smallest induced odd cycle in $H$. We claim that $G^{\bf v} =C.$  Observe that $|V(C)| \geq 2n+1.$   Consider the following decomposition  $$V(G^{\bf v})=V(H^{\bf w})= V(C)\bigsqcup \limits_{x\in V(H^{\bf w})\setminus V(C)}\{x\} .$$ Then, \begin{align*}
   \displaystyle \tau(G^{\bf v}) &= \tau(H^{\bf w})\\& \geq \tau(C)+ \sum\limits_{x\in V(H^{\bf w})\setminus V(C)}\tau({H^{\bf w}}[{\{x\}}])\\&=\frac{|V(C)|+1}{2}+\sum\limits_{x\in V(H^{\bf w})\setminus V(C)} 0\\& \geq n+1.
\end{align*} Thus, $\tau(G^{\bf v})= \tau(C)+ \sum\limits_{x\in V(H^{\bf w})\setminus V(C)}\tau({H^{\bf w}}[{\{x\}}])$ and $|V(C)|=2n+1$.  Now, if $G^{\bf v} \neq C$, then $G^{\bf v}=H^{\bf w}$ is a decomposable graph which is a contradiction. Thus, $G^{\bf v}=C$. Hence, $J=(m_C : C \text{ is an induced odd cycle in } G \text{ of size } 2n+1)$, and the assertion follows.
\end{proof}

Now, we prove Minh's conjecture for the smallest $s$ for which $I(G)^s \neq I(G)^{(s)}$. If  smallest induced odd cycle in $G$ has size $2n+1$, then by Theorem \ref{sym-decomp}, we know that $n+1$ is the smallest $s$ for which $I(G)^s \neq I(G)^{(s)}$. So, we aim to prove that $\reg\left(I(G)^{(n+1)}\right) =\reg(I(G)^{n+1})$.

\begin{lemma}\label{lower-sym-lemma}
Let $G$ be a non-bipartite graph. If $C$ is an induced odd cycle in $G$, then for all $s\geq 1,$ 
\begin{align*}
\displaystyle \reg\left(I(G)^{(s)}\right) \geq 
2s+\nu(C)+\max\{\reg(I(H_C)),1\}-2 .
\end{align*}
\end{lemma}
\begin{proof}
First, note that $H_C\sqcup C$ is an induced subgraph of $G$. By \cite[Corollary 4.5]{GHJsymbo}, $\reg\left(I(G)^{(s)}\right) \geq \reg\left(I(H_C\sqcup C)^{(s)}\right)$ for all $s\geq 1$. So, it is sufficient  to prove that 
\begin{align*}
\displaystyle \reg\left(I(H_C \sqcup C)^{(s)}\right) \geq 
2s+\nu(C)+\max\{\reg(I(H_C)),1\}-2 .
\end{align*}
 If $H_C$ is an empty graph, then $I(H_C\sqcup C)=I(C)$, and hence,   $\reg(I(H_C \sqcup C)^{(s)})=\reg(I(C)^{(s)})\geq 2s+ \nu(C)-1,$ by \cite[Theorem 4.6]{GHJsymbo}. Thus,  the assertion follows.
 
Next, we assume that $H_C$ is a non-empty graph. It follows from \cite[Theorem 5.11]{HHTT}(or  \cite[Corollary 4.6]{HJKN21}) that \begin{align*}
&\reg\left(I(H_C\sqcup C)^{(s)}\right)=\reg\left((I(H_C)+I(C))^{(s)}\right)\\&= \max_{\substack{1 \leq i \leq s-1\\ 1 \leq j \leq s}} \Big\{\reg\left(I(H_C)^{(s-i)}\right) +\reg\left(I(C)^{(i)}\right), \reg\left(I(H_C)^{(s-j+1)}\right)+\reg\left(I(H_C)^{(j)}\right)-1 \Big\}\\& \geq    \reg\left(I(H_C)\right)+\reg\left(I(C)^{(s)}\right)-1\\&\geq  2s+\nu(C)+\reg(I(H_C))-2, 
\end{align*} where the first inequality follows by using $j=s$ in previous equality, and the second inequality follows by using \cite[Theorem 4.6]{GHJsymbo}. Hence, the assertion follows.
\end{proof}

\begin{theorem}\label{main-sym-thm}
Let $G$ be a non-bipartite graph. Suppose that smallest induced odd cycle in $G$ has size ${2n+1}.$ Then, $\reg\left(I(G)^{(s)} \right) =\reg (I(G)^s)$ for $s \leq n+1.$
\end{theorem}
\begin{proof}
By Theorem \ref{sym-decomp}, $I(G)^{(s)}=I(G)^s$ for $s \leq n$. Therefore, $\reg\left(I(G)^{(s)} \right) =\reg (I(G)^s)$ for $s \leq n.$  Let $C_1,\ldots, C_r$ be the induced odd cycles of size $2n+1$ in $G$.
By Theorem \ref{sym-decomp}, $$I(G)^{(n+1)}=I(G)^{n+1}+(m_{C_i} : 1 \leq i \leq r).$$ Since $\deg (m_{C_i}) =2n+1$ and minimum degree of an element in $I(G)^{n+1}$ is $2n+2$, $m_{C_i} \notin I(G)^{n+1}$. Also, note that $|C_i|=2n+1 \leq 2(n+1)+\nu(C)-2$.  Let $V(C_i)= \{x_{i_1},\ldots,x_{i_{2n+1}}\}$ and $x \in W_{C_i}$. Then, for some $j$, $xx_{i_j} \in E(G)$ which implies that $$\displaystyle xm_{C_i}=x_{i_1}\cdots x_{i_{j-1}} x_{i_{j+1}}\cdots x_{i_{2n+1}}x x_{i_j} \in I(G)^{n+1}.$$ Next, by Lemma \ref{lower-sym-lemma}, for each $i$, we have $\reg\left(I(G)^{(n+1)}\right) \geq 
2n+\nu(C_i)+\max\{\reg(I(H_{C_i})),1\} .$ Thus, $C_1,\ldots, C_r$ satisfy the hypothesis of Theorem \ref{gen}. Hence, the assertion follows.
\end{proof}

\section{Regularity of small integral powers of edge ideals}\label{small-int}
This section studies the regularity of integral closure of small powers of edge ideals. In \cite{CV21}, it was shown that, $\displaystyle \reg\left(\overline{I(G)^{s}}\right) = \reg\left(I(G)^{s}\right)$ for $s \le 4$. We prove a more general result in this section which also generalizes the main result of \cite{CV21}.  Let $G$ be a graph. Recall that  $I(G)$ is normal if and only if $G$ is bow-free. Thus, if $I(G)^s \neq \overline{I(G)^{s}}$ for some $s$, then $G$ must contains a bow.  For a bow $B$ in $G$, the size of a bow is $\dfrac{|V(B)|}{2}$.

First, we obtain a monomial generating set of small integral powers of edge ideals. 
\begin{proposition}\label{int-decomposition}
 If the smallest size of a bow in $G$ is $k+1$, then
$$\displaystyle \overline{I(G)^s}=	\left\{\begin{array}{ll}
		I(G)^s  & \mbox{if } s \leq k \\
		I(G)^s+  \left(m_{B} \; : \; B \text{ is  a bow in } G \text{ of size } k+1\right) & \mbox{if } s = k+1.
	\end{array}
\right.$$
\end{proposition}
\begin{proof}
By \cite[Proposition 10.5.12]{villarreal_book}, we get that $$\displaystyle \overline{\mathcal{R}(I(G)t)}=\bigoplus_{s\geq 0} \overline{I(G)^s} t^s=R[I(G)t]\left[m_{B}t^{\dfrac{|V(B)|}{2}}: \; B \text{ is a bow in } G\right].$$ Let $ s$ be a positive integer. If $s \leq k$, then by comparing the graded components of degree $s$, we get $I(G)^s =\overline{I(G)^s}$. Suppose that $s= k+1$. Then, again by comparing the graded components of degree $k+1$, we have   $$\displaystyle \overline{I(G)^{k+1}}=I(G)^{k+1} +(m_{B}  \; : \; B \text{ is a bow of size } k+1).$$ Hence, the desired result follows.
\end{proof}

Let $I$ be a monomial ideal in $R$. Then, $\partial^*(I)$ denotes the ideal generated by $u/x$, where $u$ is a minimal monomial generator of $I$, and $x$ is a variable dividing $u$.
\begin{lemma}\label{tor-vanishing}
Let $I$ be a non-zero proper monomial ideal in $R$. Then, $\partial^*(\overline{I^{s}}) \subset \overline{I^{s-1}}$ for all $s \geq 1$. In particular, the inclusion map $\overline{I^s} \to \overline{I^{s-1}}$ is a Tor-vanishing map  for all $s\geq 1$, i.e., for all $i \geq 0$,  $\text{Tor}_i^R(\overline{I^s},\sk) \to \text{Tor}_i^R(\overline{I^{s-1}},\sk)$ is zero map.
\end{lemma}
\begin{proof}
Let $s$ be a positive integer, and let $f \in \partial^*(\overline{I^{s}})$ be a minimal monomial generator. Then $fx \in \overline{I^s}$ for some variable $x$ which implies that $f^a x^a \in I^{sa}$ for some positive integer $a$. Thus, there exist minimal monomial generators  $u_1,\ldots,u_{sa} $ of $I$ and a monomial $v$ in $R$ such that $f^ax^a =u_1 \cdots u_{sa} v$. Then, $f^a = \dfrac{u_1 \cdots u_{sa}v}{x^a} \in I^{sa-a}$ as $x$ cancels from at most $a$ $u_i$'s from the numerator. Consequently, we have  $f \in \overline{I^{s-1}},$ and hence, $\partial^*(\overline{I^{s}}) \subset \overline{I^{s-1}}$. The second assertion follows from \cite[Lemma 4.2 and Proposition 4.4]{NV19}
\end{proof}
In the following, we obtain the regularity formula of integral closure of powers of sum of two ideals when one ideal is normally torsion-free. 
\begin{theorem}\label{reg-tor}
Let $I$ and $J$ be squarefree monomial ideals in $R$ generated in a disjoint set of variables. If $I$ is normally torsion-free, then 
$$\reg(\overline{(I+J)^s}) = \max_{\substack{1 \leq i \leq s-1\\ 1 \leq j \leq s}} \Big\{\reg\left(I^{s-i}\right) +\reg\left(\overline{J^{i}}\right), \reg\left(I^{s-j+1}\right)+\reg\left(\overline{J^j}\right)-1 \Big\}.$$
\end{theorem}
\begin{proof}
By \cite[Theorem 2.1]{MT},  $\overline{(I+J)^s}=  \sum_{i=0}^s I^i\overline{J^{s-i}}.$ It follows from \cite[Theorem 4.5]{NV19} that for every $s \geq 1$, ${I^s} \to {I^{s-1}}$ is a Tor-vanishing map, and by Lemma \ref{tor-vanishing}, $\overline{I^s} \to \overline{I^{s-1}}$ is a Tor-vanishing map for all $ s \geq 1$. Now, applying \cite[Theorem 5.3]{HHTT},  we get $$\reg\left( \sum_{i=0}^s I^i\overline{J^{s-i}}\right)= \max_{\substack{1 \leq i \leq s-1\\ 1 \leq j \leq s}} \Big\{\reg\left(I^{s-i}\right) +\reg\left(\overline{J^{i}}\right), \reg\left(I^{s-j+1}\right)+\reg\left(\overline{J^j}\right)-1 \Big\}.$$ Hence, the assertion follows. 
\end{proof}

To prove the main result of this section, we use Theorem \ref{gen}, and to make use of Theorem \ref{gen}, we need the following lemma: 

\begin{lemma}\label{lower-int-lemma}
Let $G$ be a graph that contains at least one bow. Let $B$ be a bow in $G$. Then for all $s\geq 1,$ 
\begin{align*}
\displaystyle \reg\left(\overline{I(G)^{s}}\right) \geq 
2s+\nu(B)+\max\{\reg(I(H_B)),1\} -2.
\end{align*}
\end{lemma}
\begin{proof}
Let $e_1,\ldots,e_{\nu(B)}$ be an induced matching of $B$. Then,  $H_B\sqcup \{e_1,\ldots,e_{\nu(B)}\}$ is an induced subgraph of $G$.  If $H_B$ is an empty graph, then for all $s\geq 1$, $$\overline{I(H_B\sqcup \{e_1,\ldots,e_{\nu(B)}\})^s}=\overline{( e_1,\ldots,e_{\nu(B)})^s}=(e_1,\ldots,e_{\nu(B)})^s.$$    Thus,  by \cite[Lemma 4.4]{BHT},  $$\reg\left(\overline{I(H_B\sqcup \{e_1,\ldots,e_{\nu(B)}\})^s}\right)=2s+\nu(B)-1.$$   

Assume now that $H_B$ is a non-empty graph. Since $(e_1,\ldots,e_{\nu(B)})$ is a normally torsion-free square-free monomial ideal, by Theorem \ref{reg-tor}, \begin{align*}
&\reg\left(\overline{I(H_B\sqcup \{e_1,\ldots,e_{\nu(B)}\})^{s}}\right)=\reg\left(\overline{(I(H_B)+(e_1,\ldots,e_{\nu(B)}))^{s}}\right)\\&= \max_{\substack{1 \leq i \leq s-1\\ 1 \leq j \leq s}} \Big\{\reg\left((e_1,\ldots,e_{\nu(B)})^{s-i}\right)+\reg\left(\overline{I(H_B)^{i}}\right),\reg\left((e_1,\ldots,e_{\nu(B)})^{s-j+1}\right)+ \reg\left(\overline{I(H_B)^{j}}\right)-1 \Big\}\\& \geq    \reg\left(I(H_B)\right)+\reg\left((e_1,\ldots,e_{\nu(B)})^s\right)-1\\&= 2s+\nu(B)+\reg(I(H_B))-2, 
\end{align*} where the inequality follows by using $j=1$ in the previous equality, and the last equality follows from \cite[Lemma 4.4]{BHT}.
The rest follows from  \cite[Theorem 3.6]{BCDMS}.
\end{proof}

It follows from Proposition \ref{int-decomposition} that $\reg\left(\overline{I(G)^s}\right)=\reg(I(G)^s)$ for $s \leq k$ if the smallest size of a bow in $G$ is $k+1$. In the following, we prove the same for $s=k+1.$ 
\begin{theorem}\label{k+1}
 If the smallest size of a bow in $G$ is $k+1$, then $$\displaystyle \reg\left(\overline{I(G)^{k+1}}\right) = \reg\left(I(G)^{k+1}\right).$$ 
\end{theorem}
\begin{proof}
Let $B_1,\ldots, B_r$ be bows of size $k+1$ in $G$.  Therefore, by Proposition \ref{int-decomposition}, $\overline{I(G)^{k+1}}= I(G)^{k+1}+\left(m_{B_i} \; : \; 1 \leq i \leq r \right).$  Since $B_i$ is a bow, there exists two odd cycles $C_i$ and $C'_i$ such that $d(C_i,C'_i) \geq 2$. Therefore,  $|B_i|=2k+2 \leq 2k+2+\nu(B_i)-2$ as $\nu(B_i) =\nu(C_i)+\nu(C'_i) \geq 2.$ Also, note that   $m_{B_i} =m_{C_i}m_{C'_i}  \notin I(G)^{k+1}$ as the distance between $C_i$ and $C'_i$ is at least two. Let $x \in W_{B_i}$ be any. Then, $x \in N_G(C_i)$ or $x \in N_G(C'_i)$ which implies that  $xm_{B_i}=xm_{C_i}m_{C'_i} \in I(G)^{k+1}$. Next, by Lemma \ref{lower-int-lemma}, we have $\reg\left(\overline{I(G)^{k+1}}\right) \geq 
2k+\nu(B_i)+\max\{\reg(I(H_{B_i})),1\} .$ Hence, the assertion follows by Theorem \ref{gen}.
\end{proof}

\section{Regularity of integral closure of  powers of edge ideals of  bicyclic graphs}\label{bicyclic-int}
In this section, we study the regularity of integral closure of powers of edge ideals of bicyclic graphs. As mentioned in the introduction, one should first consider the class of odd bicyclic graphs to study the regularity function $\reg\left( \overline{I(G)^s}\right)$ for all $s \ge 1.$ In Theorem  \ref{mainTheorem}, we prove that the function $\reg\left( \overline{I(G)^s}\right)$ coincides with the regularity function of powers edge ideal when $G$ is an  odd bicyclic graph. Recall that a graph is said to be an {\it odd bicyclic graph} if it has exactly two odd cycles. 

Let $G$ be an odd bicyclic graph with odd  cycles $C_1$ and $C_2$. If the distance between the odd cycles is at most one, then $G$ is bow-free; hence, $I(G)$ is normal. Thus, we assume that $d_G(C_1,C_2) \geq 2$. Then $G$ has exactly one bow $C_1 \cup C_2$, say $B$. Also, if $C_1$ has size $2m+1$ and $C_2$ has size $2n+1$ with $m \leq n$, then the size of $B$ is $m+n+1.$

We now obtain a monomial generating set of integral closure of powers of edge ideals of odd bicyclic graphs.
\begin{proposition}\label{decomposition}
	Let $G$ be an odd bicyclic graph with the only bow $B$. Then,
$$\displaystyle \overline{I(G)^s}=	\left\{\begin{array}{ll}
		I(G)^s  & \mbox{if } s \leq m+n \\
		I(G)^s+ m_{B} I(G)^{s-m-n-1} & \mbox{if } s \geq m+n+1.
	\end{array}
\right.$$
\end{proposition}
\begin{proof}
	Since $B$ is the only  bow in $G$ and the size of the bow is $m+n+1$, by \cite[Proposition 10.5.12]{villarreal_book}, we get that $\displaystyle \overline{\mathcal{R}(I(G)t)}=R\left[I(G)t, m_{B}t^{m+n+1}\right]$.	 Let $ s$ be a positive integer. Write  $s=q(m+n+1)+r$ with $r \leq m+n$ for some  $q,r \in \mathbb{N} \cup \{0\}.$ By comparing the graded components of degree s, we have   $$\displaystyle \overline{I(G)^s}=\sum\limits_{l=0}^q m_{B}^lI(G)^{s-(m+n+1)l}.$$ 
	Now, if $ s \leq m+n$, then  $\displaystyle \overline{I(G)^s}=I(G)^s$. So, we assume that $s \geq m+n+1$. Let $V(C_1)= \{x_{i_1},\ldots,x_{i_{2m+1}}\}$. Then,  $\displaystyle m_{C_1}=x_{i_1}\cdots x_{i_{2m+1}}$. Observe that  $$\displaystyle m_{C_1}^2=(x_{i_1}x_{i_2})\cdots(x_{i_{2m-1}} x_{i_{2m}})(x_{i_{2m+1}}x_{i_{1}})\cdots  (x_{i_{2m}}x_{i_{2m+1}})\in I(G)^{2m+1}.$$ Similarly,  $\displaystyle m_{C_2}^2 \in I(G)^{2n+1}$, and hence,  $\displaystyle m_{B}^2\in I(G)^{2m+2n+2}$. Consequently,  $$\displaystyle m_{B}^lI(G)^{s-(m+n+1)l} \subset I(G)^s+ m_{B} I(G)^{s-m-n-1}$$ for $ 2 \leq l \leq q$ which completes the proof.
\end{proof}

It follows from Theorem \ref{k+1} that $\reg\left(\overline{I(G)^s}\right)=\reg(I(G)^s)$ for $s \leq m+n+1$. 
Now, we move on to prove that $\reg\left(\overline{I(G)^s}\right)=\reg\left(I(G)^s\right)$ for $s > m+n+1$, and to prove this we first compute colon ideals of the form $I(G)^s:m_{B}u_j$, where $s> m+n+1$ and $u_j$ is an element of the  minimal generating set of $I(G)^{s-m-n-1}.$
\begin{lemma}\label{quadraticLemma}
	Let $G$ be an odd bicyclic graph with the only bow $B$.   For $s > m+n+1$, let $\{u_1,\ldots,u_r\}$ be a minimal generating set of $I(G)^{s-m-n-1}$. If $m_{B}u_j \notin I(G)^{s}$, then $I(G)^s:m_{B}u_j$ is minimally generated  by  monomials of degree at most two.
\end{lemma}
\begin{proof}
	Let $x \in V(C_1)$ and $y \in V(C_2)$. Then,  $\displaystyle I(G)^{s}:m_{B}u_j =\left(I(G)^{s}:\frac{m_{B}u_j}{xy}\right):xy$. Since $\displaystyle \frac{m_{C_1}}{x} \in I(G)^m$ and $\displaystyle \frac{m_{C_2}}{y} \in I(G)^n$,  we get that $\displaystyle \frac{m_{B}}{xy}u_j \in I(G)^{s-1}$ and it is a monomial generator of $I(G)^{s-1}$. It follows from \cite[Theorem 6.1]{AB} that $\displaystyle I(G)^{s}:\frac{m_{B}u_j}{xy}$ is minimally generated by monomials of degree $2$. Thus, $\displaystyle \left(I(G)^{s}:\frac{m_{B}u_j}{xy}\right):xy$ is minimally generated by monomials of degree at most two. Hence, the assertion follows.
\end{proof}
In the following notation, we set an ordering of the minimal monomial generators of powers of edge ideals which plays an essential role in the proof of Theorem \ref{mainTheorem}. 
\begin{notation}\label{notationClosure}
	Let $G$ be an odd bicyclic graph with the only bow $B$. Then, for $ s> m+n+1$, by Proposition \ref{decomposition}, we know that $\overline{I(G)^s}=I(G)^s+m_{B}I(G)^{s-m-n-1}$.  Let $I(G)^{s-m-n-1}=( u_1,\dots, u_r) $.   By \cite[Theorem 4.12]{AB}, we assume that the ordering $u_1,\dots, u_r$ satisfies the following:  for every pair of integers $1 \leq 
	i < j \leq r$, either $(u_i : u_j) \subset  I(G)^{s-m-n} : u_j$ or  there exists an integer $k \leq i -1$ such that $(u_k : u_j)$ is generated by a
	variable, and $(u_i
	: u_j) \subset (u_k : u_j).$ Set $I_0=I(G)^s$, and for $1\leq j\leq r$, $I_j=I(G)^s+m_{B} ( u_1,\dots , u_j)$. Note that $I_r= \overline{I(G)^s}.$
\end{notation}
The next two lemmas are the most important technical part of Theorem \ref{mainTheorem}. In these lemmas, we understand the structure of ideals of the form $I_{j-1}:m_{B}u_j$, where $s> m+n+1$ and $u_j$ is an element of the  minimal generating set of $I(G)^{s-m-n-1}.$  
\begin{lemma}\label{tech2}
    Let $G$ be as in Notation \ref{notationClosure} and $u=m_{B}u_{j}$ for some $ 1 \leq j \leq r$ such that $u \notin I(G)^s$. Write $\displaystyle u_{j}=f_{1,j}\cdots f_{s-m-n-1,j}$  such that $\displaystyle f_{1,j},\dots, f_{t_{j},j}\in  E(H_B) $ and $\displaystyle f_{t_{j}+1,j},\dots, f_{s-m-n-1,j}\in E(G) \setminus E(H_B)$ for some $0 \leq t_{j}\leq s-m-n-1$.  Then $$\displaystyle I(G)^s:u  =L_{j}+I(H_B)^{t_{j}+1}:f_{1,j}\cdots f_{t_{j},j},$$ where $ L_{j}$ is an ideal generated by a subset of variables and $(w  : w \in W_B) \subset L_{j}$.
\end{lemma}
\begin{remark}
If $H_B$ is an empty graph, then the ideal  $I(H_B)^{t_{j}+1}:f_{1,j}\cdots f_{t_{j},j}$  in the above lemma is a zero ideal. Also, if $H_B$ is a non-empty graph and $t_j=0$, then    $I(H_B)^{t_{j}+1}:f_{1,j}\cdots f_{t_{j},j}$  in the above lemma is $I(H_B)$.
\end{remark}
  \begin{proof}
  Since $zm_{B}\in I(G)^{m+n+1}$  for  $z\in W_B$, we get that $zm_{B}u_{j}=zu\in I(G)^s$.   Therefore, $(w  : w \in W_B) \subset I(G)^s:u$. Let $g \in I(G)^s :u$ be a minimal monomial generator. By Lemma \ref{quadraticLemma}, we know that either $g$ is a variable or $g=y_1y_2$ for some $y_1,y_2 \in V(G)$ $(y_1$ may be equal to $y_2)$. If $g$ is a variable, then $g \in L_{j}$. Therefore, we assume that $g=y_1y_2$. Suppose that $g=y_1y_2 \in I(G)$. We claim that $g=y_1y_2 \in I(H_B) \subset I(H_B)^{t_{j}+1}:f_{1,j} \cdots f_{t_{j},j}$. Let if possible $y_1y_2 \notin I(H_B)$. Then either $y_1 \in W_B$ or $y_2 \in W_B$ which implies that $y_1 \in I(G)^s: u$ or $y_2 \in I(G)^s:u$. In either case, we are contradictory to the fact that $g=y_1y_2$ is a minimal monomial generator of $I(G)^s:u.$ Thus, the claim follows. Assume now that $g=y_1 y_2 \notin I(G).$ 
  Let $x \in V(C_1)$ and $y \in V(C_2)$. Then,  $$\displaystyle I(G)^s:u=I(G)^{s}:m_{B}u_{j} =\left(I(G)^{s}:\frac{m_{B}u_{j}}{xy}\right):xy.$$ Therefore, $\displaystyle gxy \in I(G)^{s}:\frac{m_{B}u_{j}}{xy}$.   Since  $ u \notin I(G)^s$, we get $\displaystyle xy \notin I(G)^s :\frac{m_{B}u_{j}}{xy} $.   Observe that if $\displaystyle y_ix \;(\text{ or } y_iy) \in I(G)^s: \frac{m_{B}u_{j}}{xy}$, then $y_i \in I(G)^s:u$ which is a contradiction to the fact that $g=y_1y_2$ is a minimal monomial generator of $I(G)^s:u$.  Therefore, we get that $\displaystyle xy,y_1x,y_2x,y_1y,y_2y \notin I(G)^s :\frac{m_{B}u_{j}}{xy} $, and hence, $\displaystyle y_1y_2 \in  I(G)^s :\frac{m_{B}u_{j}}{xy} $ is a minimal monomial generator. Let $f_{s-m-n,j}, \ldots, f_{s-n-1,j}$ denote the edges of $C_{2m+1}\setminus \{x\}$ such that $\dfrac{m_{C_1}}{x}= f_{s-m-n,j} \cdots f_{s-n-1,j}$, and $f_{s-n,j}, \ldots, f_{s-1,j}$ denote the edges of $C_{2n+1}\setminus \{y\}$ such that $\dfrac{m_{C_2}}{y}= f_{s-n,j} \cdots f_{s-1,j}$. Therefore,  $\dfrac{m_{B}u_{j}}{xy}=\dfrac{m_{B}}{xy}u_{j} =f_{1,j} \cdots f_{s-1,j}$. Now, by \cite[Theorem 6.7]{AB},  $y_1$ is even-connected to  $y_2$ with respect to  $\displaystyle f_{1,j} \cdots f_{s-1,j}$.   Consequently,  there exists a walk $P: \; p_0,\ldots,p_{2k+1}$ for some $k \geq 1$ in $G$  such that \begin{itemize}
		\item $p_0 =y_1$ and $p_{2k+1}=y_2$.
		\item for all $ 1 \leq l \leq k$, $p_{2l-1}p_{2l} =f_{i,j}$ for some $i$.
		\item for all $i$, $|\{l \geq 1 : p_{2l-1}p_{2l} =f_{i,j}\}| \leq |\{ l : f_{l,j} =f_{i,j}\}|$.
	\end{itemize} If the walk $P: \;p_0,\ldots,p_{2k+1}$  contains a vertex of $W_B$, then there exists $z \in W_B$ such that $p_t =z$ for some  $t \leq 2k+1$.  Assume, without loss of generality, that  $z \in N_G(C_1)$. We claim that $y_1$ and $x$ are even-connected with respect to $f_{1,j} \cdots f_{s-1,j}$.  Among {the} two walks form $z$ to $x$ along edges of cycle $C_1$ one gives an even-connection between $y_1$ and $x$ with respect to  $f_{1,j} \cdots f_{s-1,j}$. Thus, $y_1$ and $x$ are even-connected with respect to $f_{1,j} \cdots f_{s-1,j}$. By \cite[Theorem 6.7]{AB}, $\displaystyle y_1x \in I(G)^s:f_{1,j} \cdots f_{s-1,j}$. Since $\displaystyle I(G)^{s}:u =\left(I(G)^{s}:f_{1,j} \cdots f_{s-1,j}\right):xy$, we have  $y_1 \in I(G)^s:u$ which is a contradiction to the fact that $g=y_1y_2$ is a minimal monomial generator of $I(G)^s:u$.   Therefore,  $V(P) \subset V(H_B)$.  Observe that the walk $P: \;p_0,\ldots,p_{2k+1}$ in $H_B$ satisfies: \begin{itemize}
		\item $p_0 =y_1$ and $p_{2k+1}=y_2$.
		\item for all $ 1 \leq l \leq k$, $p_{2l-1}p_{2l} =f_{i,j}$ for some $1 \leq i \leq t_j$.
		\item for all $1\leq i \leq t_j$, $|\{l \geq 1 : p_{2l-1}p_{2l} =f_{i,j}\}| \leq |\{ l : 1 \leq l \leq t_j \text{ and } f_{l,j} =f_{i,j}\}|$,
	\end{itemize} which implies that $y_1$ and $y_2$ are even-connected in $H_B$ with respect to $f_{1,j} \cdots f_{t_j,j}$. Thus,  by \cite[Theorem 6.7]{AB}, $y_1y_2 \in   I(H_B)^{t_j+1}:f_{1,j} \cdots f_{t_j,j},$ and hence, $I(G)^s:u \subset L_{j}+I(H_B)^{t_j+1}:f_{1,j} \cdots f_{t_j,j}$. 
	
	Conversely, let $m$ be a minimal monomial generator of $I(H_B)^{t_j+1}:f_{1,j}\cdots f_{t_j,j}$.  Since $ \displaystyle \frac{u}{(f_{1,j} \cdots f_{t_j,j})} =m_{B} \frac{u_{j}}{(f_{1,j} \cdots f_{t_j,j})} \in I(G)^{s-t_j-1}$, we have $$\displaystyle mu= m (f_{1,j} \cdots f_{t_j,j}) \frac{u}{(f_{1,j} \cdots f_{t_j,j})} \in I(H_B)^{t_j+1} I(G)^{s-t_j-1} \subset  I(G)^s$$ which further implies that $m \in I(G)^s:u$. Hence, the assertion follows.
  \end{proof}

\begin{lemma}\label{colonLemma}
	Let $G$ be as in Notation \ref{notationClosure} and $u=m_{B}u_{j}$ for some $ 2 \leq j \leq r$ such that $u \notin I(G)^s$. Write $u_{j}=f_{1,j}\cdots f_{s-m-n-1,j}$  such that $f_{1,j},\dots, f_{t_j,j}\in  E(H_B) $ and $f_{t_j+1,j},\dots, f_{s-m-n-1,j}\in E(G) \setminus E(H_B)$ for some $0 \leq t_j\leq s-m-n-1$. Then, $$\displaystyle I_{j-1}:u=I(G)^s:u +(m_{B}u_1,\ldots,m_{B}u_{j-1}):u=L_{j}+I(H_B)^{t_j+1}:f_{1,j} \cdots f_{t_j,j},$$ where $ L_{j}$ is an ideal generated by a subset of variables and $(w  : w \in W_B) \subset L_{j}$. 
\end{lemma}
\begin{proof}
 By Notation \ref{notationClosure}, for any $i \leq j-1$, either $(u_i :u_{j}) \subset I(G)^{s-m-n}:u_{j}$ or there exists $k \leq j-1$ such that $(u_k:u_{j})$ is generated by a 
	variable, and $(u_i :u_{j}) \subset (u_k:u_{j})$. Thus, for any $i \leq j-1$,  either $(m_{B}u_i :u)=(u_i : u_{j}) \subset I(G)^{s-m-n}:u_{j}$ or there exists $k \leq j-1$ such that $(m_{B}u_k:u)$ is generated by a
	variable, and $(m_{B}u_i :u) \subset (m_{B}u_k:u)$. Therefore, $(m_{B}u_1,\ldots,m_{B}u_{j-1}):u\subset {L}_{1,j}+I(G)^{s-m-n}:u_{j}$, where $L_{1,j}$ is an ideal generated by a subset of variables. Note that $I(G)^{s-m-n}:u_{j}=m_{B}I(G)^{s-m-n}:u \subset I(G)^s :u$ as $m_{B}  \in I(G)^{m+n}$. Hence, we have $I_{j-1}:u=I(G)^s:u +(m_{B}u_1,\ldots,m_{B}u_{j-1}):u= L_{1,j} +I(G)^s:u$. 
 
 Now, by Lemma \ref{tech2},   $I(G)^s:u  =L_{2,j}+I(H_B)^{t_j+1}:f_{1,j} \cdots f_{t_j,j},$ where $ L_{2,j}$ is an ideal generated by a subset of variables and $(w  : w \in W_B) \subset L_{2,j}$. Take $L_{j} = L_{1,j}+L_{2,j}$. Then, $L_{j} $ is an ideal generated by subsets of variables  and $(w  : w \in W_B) \subset L_{j}$. Hence, the assertion follows.
\end{proof}

We now prove the main result of this section and conclude the article.

\begin{theorem}\label{mainTheorem}
	Let $G$ be a graph as in Notation \ref{notationClosure}. Then  $$\reg\left(\overline{I(G)^s}\right)= \reg\left(I(G)^s\right) \text{ for all } s>m+n+1.$$
\end{theorem}

\begin{proof}
Fix $s > m+n+1$. Recall that $I_0=I(G)^s$ and $I_j=I(G)^s+m_{C_1}m_{C_2}(u_1,\ldots,u_j)$ for $1 \leq j \leq r$.	We claim that  $\reg(I_0)=\reg(I_{r})$.  For $ 1\leq j \leq r$, consider the following short exact sequences 
\begin{align}\label{ses2}
  0\longrightarrow {I_{j-1}:m_{C_1}m_{C_2}u_{j}}(-2s)\xrightarrow{\cdot m_{C_1}m_{C_2}u_{j}}  {I_{j-1}}\longrightarrow {I_{j}}\longrightarrow 0.  
\end{align}	
If $m_{C_1}m_{C_2}u_{j} \in I(G)^s$, then $I_{j-1} : m_{C_1}m_{C_2}u_{j} =R$. Therefore, $\reg\left( {I_{j-1}:m_{C_1}m_{C_2}u_{j}}(-2s)\right)=2s.$ Suppose that $m_{C_1}m_{C_2}u_{j} \notin I(G)^s$.	Write $u_{j}=f_{1,j}\cdots f_{s-m-n-1,j}$  such that $f_{1,j},\dots, f_{t_j,j}\in  E(H_B) $ and $f_{t_j+1,j},\dots, f_{s-m-n-1,j}\in E(G) \setminus E(H_B)$ for some $0 \leq t_j\leq s-m-n-1$. Thus, by Lemma \ref{colonLemma}, $I_{j-1}:m_{C_1}m_{C_2}u_{j}=L_{j}+I(H_B)^{t_j+1}:f_{1,j}\cdots f_{t_j,j}$, where $ L_{j}$ is an ideal generated by a subset of variables and $(w  : w \in W) \subset L_{j}$. If  $H_B$ is non-empty, then it follows from the proof of  \cite[Theorem 1.1 (ii)]{AE20} that $\reg\left( I(H_B)^{t_j+1}:f_{1,j} \cdots f_{t_j,j} \right) \leq \reg(I(H_B))$ as $H_B$ is a bipartite graph. Therefore, for each $1 \leq j \leq r$, \begin{align*}
	   & \reg\left( {I_{j-1}:m_{C_1}m_{C_2}u_{j}}(-2s)\right)\\&=
	    \left\{\begin{array}{cc}
	        2s+\reg\left( L_j+I(H_B)^{t_j+1}:f_{1,j} \cdots f_{t_j,j} \right) & \mbox{if } m_{C_1}m_{C_2}u_{j} \notin I(G)^s \\
	     2s    & \; \mbox{if } m_{C_1}m_{C_2}u_{j} \in I(G)^s
	    \end{array}
	    \right.
	  \\&  \leq 2s+\max\{\reg(I(H_B)), 1\}\\ & \leq 2s + \nu(B)+ \max\{\reg(I(H_B)), 1\}-2,
	\end{align*} as $\nu(B) \geq 2$. Suppose that $ \reg\left(I(G)^{s}\right) > 2s +\nu(B)+\max\{\reg(I(H_B)), 1\}-2. $ Then, by recursively  applying  Lemma \ref{reg-lem} on short exact sequences \eqref{ses2}, we get  \begin{align*}
	  \reg\left(I_0\right)=\reg\left(I_{1}\right)=\cdots =\reg\left(I_r\right).  
	\end{align*} 
	
	Next, suppose that $ \reg\left(I(G)^{s}\right) \leq  2s +\nu(B)+\max\{\reg(I(H_B)), 1\}-2. $ Then, by Lemma \ref{new-lower-lem}, that  $ \reg\left(I(G)^{s}\right) =  2s +\nu(B)+\max\{\reg(I(H_B)), 1\}-2. $ By applying Lemma \ref{reg-lem} on short exact sequences \eqref{ses2}, we get
\begin{align*}
    \reg\left(I_r\right) & \leq \max\Big\{ \reg\left(I_{r-1}\right), \reg\left( I_{r-1}:m_{C_1}m_{C_2}u_r (-2s) \right)-1 \Big\} \\ & \leq \max\Big\{ \reg\left(I_{r-2}\right), \reg\left( I_{r-2}:m_{C_1}m_{C_2}u_{r-1} (-2s) \right)-1, \reg\left( I_{r-1}:m_{C_1}m_{C_2}u_r(-2s) \right)-1 \Big\} \\ & \leq \cdots \cdots \cdots \cdots \cdots \cdots (\text{ continuing this process })\\ & \leq \max\Big\{ \reg\left(I_{0}\right), \reg\left( I_{j-1}:m_{C_1}m_{C_2}u_{j}(-2s) \right)-1 \text{ for } 1 \leq j\leq  r \Big\} \\ & \leq 2s+ \nu(B)+\max\{\reg(I(H_B)), 1\} -2.
\end{align*} By Lemma \ref{lower-int-lemma},  $2s+\nu(B)+\max\{\reg(I(H_B)), 1\} -2 \leq \reg\left(\overline{I(G)^{s}}\right)=\reg\left(I_r\right)\leq 2s+\nu(B)+\max\{\reg(I(H_B)), 1\}-2,$ which implies that  $\displaystyle\reg\left(\overline{I(G)^{s}}\right)=\reg\left(I(G)^{s}\right).$   Hence, the assertion follows.
\end{proof}

%\noindent
%{\bf Data Availability:} Data sharing does not apply to this article as no data sets were generated or analyzed during the study.

\bibliographystyle{plain}
\bibliography{refs_reg}
\end{document}